\documentclass[a4paper]{article}

\usepackage{amsfonts}
\usepackage{amsmath}
\usepackage{graphicx}
\usepackage{latexsym}
\usepackage{amsthm}
\usepackage{mathrsfs}
\usepackage{enumitem}
\usepackage{graphpap}

\DeclareMathOperator{\aff}{aff}

\theoremstyle{plain}
\newtheorem{theorem}{Theorem}[section]
\newtheorem{definition}[theorem]{Definition}
\newtheorem{corollary}[theorem]{Corollary}
\newtheorem{lemma}[theorem]{Lemma}
\newtheorem{prop}[theorem]{Proposition}
\newtheorem*{remark}{Remark}

\date{\today}

\begin{document}

\title{How to sew in practice?}

\author{ R. Trelford,  V. V\'{\i}gh\footnote{The author was partially supported by Hungarian OTKA Grant T75016.}}


\maketitle

\begin{abstract}
	In this note we prove a theorem concerning the sewing of even dimensional neighbourly polytopes (see \cite{shem}). The theorem provides a fast algorithm for sewing in practice. We also give a description of the universal faces of a sewn $d$-polytope in terms of the main theorem.
\end{abstract}

\section{Introduction}
	Let  $\mathbb{E}^d$ be the $d$-dimensional Euclidean space, and write $[A]$ for the convex hull of a point set $A\subset \mathbb E^d$. Let $\{v_1\ldots,v_n\}$ be a finite point set in $\mathbb E^d$, and let $P=[v_1,\ldots,v_n]$. We say that $P$ is a convex polytope, and the the dimension of $P$ is the dimension of $\aff P$, the affine hull of $P$.  Assume $P$ is a $d$-dimensional polytope in $\mathbb E^d$, let $H$ be a supporting hyperplane of $P$, and let $G=H\cap P$.  Then $G$ is a proper face of $P$, which is itself a polytope.  We call a $0$-dimensional face a vertex, a $1$-dimensional face an edge, and a $(d-1)$-dimensional face a facet.  We denote by $\mathcal{V}(P)$ the set of vertices of $P$, $\mathcal{F}(P)$ the set of all facets of $P$, and more generally, $\mathcal{F}_j(P)$ is the set of all $j$-dimensional faces of $P$ for $0\leq j\leq d-1$.  We set $\mathcal{B}(P)=\bigcup_{j=0}^{d-1} \mathcal F_j \cup \{\emptyset\}$, the boundary complex of $P$. If $G$ is a face of $P$, we denote by $P/G$ the quotient polytope of $P$ with respect to $G$.  We treat $P/G$ as a polytope; for details see \cite{mcmushep}. 

	A $d$-dimensional convex polytope $P$ is $k$-neighbourly if every $k$ vertices determine a proper face $F$ of $P$. The $\lfloor d/2 \rfloor$-neighbourly polytopes are called neighbourly polytopes. The most widely known examples of neighbourly polytopes are the cyclic polytopes (see \cite{grunb}). With the celebrated Upper Bound Theorem of P. McMullen \cite{mcmu}, neighbourly polytopes have become the subject of special interest, as it was shown that among all $d$-polytopes with $n$ vertices, the neighbourly polytopes have the maximal number of $j$-dimensional faces for $1\leq j\leq d-1$.  However, it remained a challenging task to construct infinite classes of neighbourly polytopes other than the cyclic polytopes. In 1981 I. Shemer \cite{shem}  introduced the concept of sewing which produced a new infinite class of even dimensional neighbourly polytopes that contains the class of cyclic polytopes. In 2001 T. Bisztriczky \cite{bisztr} extended Shemer's method for odd dimensional simplicial neighbourly polytopes. A further generalization was obtained very recently by C. Lee and M. Menzel \cite{leemen}. 


\section{The sewing construction}

	Our results concern the original sewing process of I. Shemer \cite{shem}. We briefly review the definition of sewing and we recall some results. From now on $P$ will always denote $(2m)$-dimensional neighbourly polytope $(m>1)$ with at least $2m+3$ vertices. It is well known (see \cite{grunb}) that $P$ is simplicial. 

	\begin{definition}[\cite{shem}, Definition $3.2, 3.2^*$]\label{universalface}
		A $k$-face $U$ of $P$ is {\em universal} if either the quotient polytope $P/U$ is a neighbourly polytope with $|\mathcal{V}(P)|-k-1=|\mathcal{V}(P)|-|\mathcal{V}(U)|$ vertices or if $U$ is a facet of $P$. Equivalently $U$ is a {\em universal} $k$-face of $P$ if $[S, U]$ is a face of $P$ for all $S\subset \mathcal{V}(P)$ with $|S|\leq \lfloor (2m-k-1)/2 \rfloor$. We denote by $\mathcal U_k(P)$ the set of all universal $k$-faces of $P$, and by $\mathcal{U}(P)$ the set of all universal faces of $P$.
	\end{definition}

	We put here a proposition about universal faces that will be useful later.

	\begin{prop}\label{univquot}
		Let $P$ be a neighbourly $(2m)$-polytope and assume that $U\in \mathcal U_k(P)$, and $U\subset V\in \mathcal F_n(P)$, with $0\leq k<n\leq 2m-1$. Then $V\in \mathcal U_n(P)$ if, and only if, $V/U\in \mathcal U_{n-k-1}(P/U)$.
	\end{prop}

	\begin{proof}
		From the definition of quotient polytopes it follows that \[P/V\cong (P/U)/(V/U).\]  The left handside is a neighbourly polytope with $|\mathcal V(P)|-n-1$ vertices if, and only if, the right handside is. This proves the claim by Definition \ref{universalface}.
	\end{proof}

	Let $x_1, x_2, \ldots, x_m$ and $y_1, y_2, \ldots, y_m$ be distinct vertices of $P$, and define $\Phi_1=[x_1, y_1]$ and  $\Phi_j=[\Phi_{j-1}, x_j, y_j]$ for $2\leq j\leq m$.  $\mathcal T=\{\Phi_1, \ldots, \Phi_m\}$ is a {\em universal tower} in $P$ if $\Phi_j\in \mathcal U_{2j-1}(P)$ for all $j=1,\ldots,m$. We denote by $\mathscr F_i$ the set of facets of $P$ that contain $\Phi_i$ with the additional convention that $\mathscr F_{j}=\emptyset$ if $j>m$ and $\mathscr F_0=\mathcal F(P)$, the set of all facets. Furthermore let $\mathcal C(\mathcal T)=(\mathscr F_1 \backslash \mathscr F_2) \cup (\mathscr F_3 \backslash \mathscr F_4) \cup \ldots$. Let $F$ be a facet of $P$. We say that the point $x\notin \aff F$ is beyond $F$ (with respect to $P$) if $\aff F$ separates $P$ and $x$, otherwise $x$ is beneath $F$. The point $x$ is exactly beyond $\mathcal C(\mathcal T)$ if it is beyond $F$ for every $F\in\mathcal C(\mathcal T)$, and beneath $F$ for every   $F\in\mathscr F_0 \backslash \mathcal C(\mathcal T)$. Lemma 4.4 in \cite{shem} states that for every universal tower $\mathcal T$ there exists a point $\bar x=\bar x(\mathcal T)$ that lies exactly beyond $\mathcal C(\mathcal T)$.

	Figure \ref{torony} helps to keep in mind the structure of the beyond and beneath facets.  In the first category are those facets of P that do not contain the sewing edge $\Phi_1$; $\bar{x}$ is beneath these facets. In the next category are those facets of $P$ that contain the sewing edge $\Phi_1$, but do not contain the sewing universal $3$-face $\Phi_2$; $\bar{x}$ is beyond these facets, and so on.  Note that the universal faces $\Phi_1,\ldots,\Phi_m$ define the categories, but only $\Phi_m$ is contained in any of the categories.
	
	\begin{figure}
		\begin{picture}(300,150)(0,0)
			\put(75,200){\line(1,0){150}}
			\put(75,150){\line(1,0){150}}
			\put(75,100){\line(1,0){150}}
			\put(75,50){\line(1,0){150}}
			\put(75,0){\line(0,1){200}}
			\put(225,0){\line(0,1){200}}
			\put(100,175){\text{$\Phi_0$}}
			\put(175,175){\text{beneath}}
			\put(240,175){\text{Category I}}
			\put(100,125){\text{$\Phi_1$}}
			\put(175,125){\text{beyond}}
			\put(240,125){\text{Category II}}
			\put(100,75){\text{$\Phi_2$}}
			\put(175,75){\text{beneath}}
			\put(240,75){\text{Category III}}
			\put(150,25){\text{$\vdots$}}
		\end{picture}
	\caption{Beyond and beneath facets}\label{torony}
	\end{figure}

	
	The following theorem is the main result of \cite{shem}. Its importance lies in the fact that it allows one to construct infinite families of neighbourly polytopes that are not cyclic.

	\begin{theorem}[\cite{shem}, Theorem 4.6]\label{sewnpoly}
		Let $P$ be a neighbourly $2m$-polytope, $\mathcal T$ is a universal tower of $P$, and assume that $\bar x$ lies exactly beyond $\mathcal C(\mathcal T)$. Then $P^+=[\bar x, P]$ is a neighbourly $2m$-polytope, and $\mathcal V(P^+)=\mathcal V(P)\cup\{\bar x\}$. We say that $P^+$ is obtained by sewing the vertex $\bar x$ onto the polytope $P$ through the tower $\mathcal T$.
	\end{theorem}

	The following theorem shows how $\mathcal{T}$ behaves in the sewn polytope, $P^+$:
	\begin{theorem}[\cite{shem}, Theorem 4.6]\label{towerleftovers}
		Let $P^+=[P, \bar x]$ be obtained by sewing $\bar x$ onto $P$ through the tower $\mathcal T$.  Then
		\begin{enumerate}
			\item\label{evenstaysuniversal} If $0<j\leq m$ is even, then $\Phi_j$ is a universal face of $P^+$.
			\item\label{oddstaysface} If $0<j\leq m$ is odd, then $\Phi_j$ is \textbf{not} a universal face of $P^+$, but if $j<m$ then $\Phi_j$ is a face of $P^+$.
			\item $[\Phi_{j-1},x_j,\bar x]$ is a universal face of $P^+$ for $1\leq j\leq m$.
		\end{enumerate}
	\end{theorem}


	I. Shemer \cite{shem} also described all the universal faces of $P^+$ in terms of missing faces:
	\begin{definition}[\cite{shem}, Definition 4.1]\label{missingface}
		If $G\in\mathcal{B}(P)$ and $M\subseteq\mathcal{V}(P)\setminus\mathcal{V}(G)$, then we say that $M$ is a missing face of $P$ relative to $G$ if $[M,G]\not\in\mathcal{B}(P)$, but $[M',G]\in\mathcal{B}(P)$ for every $M'\subset M$.  We define:
		\begin{align*}
			\mathcal{M}(P/G) &= \{M:M\mbox{ is a missing face of $P$ relative to $G$}\},\\
			\mathcal{M}(P) &= \mathcal{M}(P/\emptyset);\mbox{ the set of all missing faces of $P$.}\\
		\end{align*}
	\end{definition}

	\begin{lemma}[\cite{shem}, Lemma 4.7]\label{sewnmissingface}
		Let $P^+=[P, \bar x]$ be obtained by sewing $\bar x$ onto $P$ through the tower $\mathcal T$.  If $M\subset\mathcal{V}(P^+)$, then $M\in\mathcal{M}(P^+)$ if and only if either
		\begin{enumerate}
			\item\label{notcontainx} $M=\left(\bigcup_{i=1}^j\{x_{2i-1},y_{2i-1}\}\right)\cup A$ for some integer $0\leq j\leq\frac{m+1}{2}$ and some $A\in\mathcal{M}(P/\Phi_{2j})$, or
			\item\label{containx} $M=\left(\bigcup_{i=1}^j\{x_{2i},y_{2i}\}\right)\cup A\cup\{\bar x\}$ for some integer $0\leq j\leq\frac{m}{2}$ and some $A\in\mathcal{M}(P/\Phi_{2j+1})$.
		\end{enumerate}
		(Recall that $\Phi_{m+1}=P$ and $\mathcal{M}(P/P)=\{\emptyset\}$).
	\end{lemma}

	Finally, the connection between universal faces and missing faces is:
	\begin{prop}[\cite{shem}, Proposition 4.2(8)]\label{missuniverse}
		$U$ is a universal $(2k-1)$-face of a neighbourly $2m$-polytope $Q$ if and only if $|M\cap U|\leq k$ for every $M\in\mathcal{M}(Q)$.
	\end{prop}

	For completeness, we also include a result from \cite{grunb}:
	\begin{theorem}[\cite{grunb}, Theorem 5.2]\label{bbp}
		Let $P^+=[P, \bar x]$ be obtained by sewing $\bar x$ onto $P$ through the tower $\mathcal T$. For all facets $F$ of $P^+$ exactly one of the following holds:
		\begin{enumerate}[label=(\alph{enumi})]
			\item $F \in  \mathcal F_0 \backslash \mathcal C(\mathcal T)$, that is, $\bar x$ is beneath $F$.
			\item $F= [\bar x, G]$ where $G$ is a $(2m-2)$-face of $P$ such that there exist facets $F_1\supset G$ and $F_2\supset G$ of $P$ with $F_1 \in  \mathcal F_0 \backslash \mathcal C(\mathcal T)$ ($\bar x$ is beneath $F_1$) and $F_2 \in  \mathcal C(\mathcal T)$ ($\bar x$ is beyond $F_2$). In this case we say that $G$ has the beyond--beneath property (BBP).
		\end{enumerate}
	\end{theorem}


\section{Main result}\label{mainsection}

	Let $P$ be a neighbourly $(2m)$-polytope, and $\mathcal T$ be a universal tower of $P$. Assume $P^+=[P, \bar x]$ is obtained by sewing $\bar x$ onto $P$ through $\mathcal T$. From Theorem \ref{towerleftovers} it follows that for $1\leq i \leq m$, $[\Phi_{i-1}, x_i, x]$ is a universal face of $P^+$. Consider the quotient polytope $P/\Phi_i$ which is a neighbourly $(2m-2i)$-polytope by definition. To every vertex $v\in \mathcal{V}(P)\setminus\mathcal{V}(\Phi_i)$ there corresponds a vertex $v^*$ of $P/\Phi_i$, since $\Phi_i$ is universal, and $m>1$. Let $n>i$ and consider the face $\Phi_n^*$ of $P/\Phi_i$ that corresponds to $\Phi_n$. This $\Phi_n^*$ exists since $\Phi_n\supset \Phi_i$. From Proposition \ref{univquot} it follows that $\Phi_n^*$ is a universal face of $P/\Phi_i$. We denote by $\mathcal T^ *$ the universal tower of $P/\Phi_i$ corresponding to $\mathcal T$, that is, $\mathcal T^ *$ consists of all $\Phi_n^*$ with $n>i$. Let $(P/\Phi_i)^+=[P/\Phi_i, \bar y^*]$ be obtained by sewing $\bar y^*$ onto $P/\Phi_i$ through $\mathcal T^*$.  We are in position to state our main theorem.

	\begin{theorem}\label{main}
		With the notion above we have that for $1\leq i\leq m$, \[(P/\Phi_i)^+\cong P^+/[\Phi_{i-1}, x_i, \bar x],\] where the bijection $\varphi$ of the vertices is given by $v^* \mapsto v^{**}$ if $v\in\mathcal{V}(P)\setminus\mathcal{V}(\Phi_i)$, and $\bar y^*\mapsto y_i^{**}$. For $v\in\mathcal V(P^+)\backslash(\mathcal{V}(\Phi_{i-1})\cup \{x_i, \bar x\})$, $v^{**}$ denotes the corresponding vertex of $P^+/[\Phi_{i-1}, x_i, \bar x]$. 
	\end{theorem}

	The importance of this theorem is shown in Section \ref{algo}: it allows us to reduce a sewing in $(2m)$-dimensions to a trivial sewing in $2$-dimensions. The key observation is that while the sewing on the left handside in Theorem \ref{main} is in $(2m-2i)$-dimensions, the sewing on the right handside is in $(2m)$-dimensions. The following proposition will be useful in the proof of Theorem \ref{main}:

	\begin{prop}\label{tower}
		Let $F^*=[v_1^*, \ldots , v_{2m-2i}^*]$ be a facet of $P/\Phi_i$.  Then $\bar y^*$ is beneath (beyond) $F^*$ if and only if either
		\begin{enumerate}
			\item $\bar x$ is a beyond (beneath) $[\Phi_i,F]$ and $i$ is even, or
			\item $\bar x$ is beneath (beyond) $[\Phi_i,F]$ and $i$ is odd.
		\end{enumerate}
	\end{prop}
	
	\begin{proof}
		Observe that after taking the quotient polytope with respect to $\Phi_i$, Figure \ref{torony} ``shifts $i$ steps down.''  Both implications follow.
	\end{proof}

	\begin{proof}[Proof of Theorem \ref{main}]
		Note that $(P/\Phi_i)^+$ and  $P^+/[\Phi_{i-1}, x_i, \bar x]$ are simplicial. It follows that it is enough to prove that $[v_1^*,\ldots,v_{2m-2i}^*]$ is a facet of $(P/\Phi_i)^+$ if and only if $[v_1^{**},\ldots,v_{2m-2i}^{**}]$ is a facet of $P^+/[\Phi_{i-1}, x_i, \bar x]$ ($\bar y^{**}=y_i^{**}$).  We start by examining the facets of $P^+/[\Phi_{i-1},  x_i, \bar x]$. For simplicity let $G=[\Phi_{i-1}, x_i]$, and $k=2m-2i$. Now $[v_1^{**}, \ldots , v_k^{**}]$ is a facet of $P^+/[G, \bar x]$ if and only if $[G,\bar x, v_1, \ldots, v_k]$ is a facet of $P^+$. All of these facets of $P^+$ are of type (b)  (see Theorem \ref{bbp}). This means that $[G, \bar x, v_1, \ldots , v_k]$ is a facet of $P^+$ if and only if $[G, v_1, \ldots , v_k]$ is a $2m-2$ face of $P$ with the BBP. We examine two cases according to whether $y_i\in \{v_1, \ldots, v_k\}$ or not.

		{\bf Case I.}  $y_i\notin \{v_1, \ldots, v_k\}.$\\ 
		First note that $[G, v_1, \ldots , v_k]$ is contained in exactly two facets of $P$. Since $[G, v_1, \ldots , v_k]$ has the BBP, we may assume that $\bar x$ is beneath $[w, G, v_1, \ldots , v_k]$ and beyond $[v, G, v_1, \ldots , v_k]$ ($v\neq w$). It is easy to see that neccesarily $v=y_i$ if $i$ is odd and $w=y_i$ if $i$ is even, that is, $\bar x$ is beyond $[y_i, G, v_1, \ldots , v_k]$  if $i$ is odd and $\bar x$ is beneath $[y_i, G, v_1, \ldots , v_k]$ if $i$ is even. Similarly, it can be seen that for every facet $F=[y_i, G, w_1, \ldots, w_k]$ of $P$ we have that $F'= [G, w_1, \ldots, w_k]$ is a $(2m-2)$-face of $P$ with the BBP. In summary $[G, v_1, \ldots , v_k]$ is a $(2m-2)$-face of $P$ with the BBP if and only if $\bar x$ is beyond $[y_i, G, v_1, \ldots , v_k]$ and $i$ is even, or $\bar x$ is beneath $[y_i, G, v_1, \ldots , v_k]$ and $i$ is odd.

		Since all the facets of $P/\Phi_i$ that $\bar x$ is beneath correspond to type (a) facets of $(P/\Phi_i)^+$ (see Theorem \ref{bbp}), Proposition \ref{tower} implies that if $y_i\notin \{v_1, \ldots, v_k\}$ then $[G, v_1, \ldots , v_k]$ is a $(2m-2)$-face of $P$ with the BBP if and only if $[v_1^*, \ldots , v_{k}^*]$ is a type (a) facet of  $(P/\Phi_i)^+$.  This concludes Case I.

		{\bf Case II.} $y_i\in \{v_1, \ldots, v_k\}.$\\
		We would like to prove that $[G, y_i, v_1, \ldots , v_{k-1}]$ is a $(2m-2)$-face of $P$ with the BBP if and only if $[\bar y^*, v_1^*, \ldots , v_{k-1}^*]$ is a (type (b)) facet of $(P/\Phi_i)^+$. Since $[G, y_i, v_1, \ldots , v_{k-1}]$ is a $(2m-2)$-face of $P$ with the BBP, there exist $v, w \in \mathcal V(P)$ such that $\bar x$ is beneath $[v, G, y_i, v_1, \ldots , v_{k-1}]$  and beyond $[w, G, y_i, v_1, \ldots , v_{k-1}]$  of $P$. Using Proposition \ref{tower}, we obtain that $\bar x$ is beyond one of $[v^*, v_1^*, \ldots , v_{k-1}^*]$ and $[w^*, v_1^*, \ldots , v_{k-1}^*]$ and beneath the other. This implies that  $[v_1^*, \ldots , v_{k-1}^*]$ is a $(2m-2i-2)$-face of $P/\Phi_i$ with the BBP, and from Theorem \ref{bbp} it follows that $[\bar y^*, v_1^*, \ldots , v_{2m-3}^*]$ is a type (b) facet of $(P/\Phi_i)^+$. 

		This completes the proof of Theorem \ref{main}.
	\end{proof}

	\begin{remark}
		In Theorem \ref{main}, we may interchange the roles of $x_i$ and $y_i$. 
	\end{remark}

\section{Sewing in practice}\label{algo}

	In this section we present an algorithm for sewing in practice. If the dimension is fixed, then this algorithm is the best possible, that is, it has linear running time in the number of facets of $P$. Since $P^+$ has more facets than $P$ we cannot expect better than this. The algorithm is based on the special case of Theorem \ref{main} when $i=1$.

	\begin{corollary}\label{main1} With the above notation
		\[(P/\Phi_1)^+\cong P^+/[x_1,\bar x] \cong P^+/[y_1,\bar x].\]
	\end{corollary}
 
	Note that we assume in this section that the intitial polytope $P$ is given by the list of its facets.  We make some preliminary remarks: First we note that all the information of the face-lattice of $P$ is contained in the list of the facets and since $P$ is simplicial, it is trivial to derive the face lattice of $P$ from this list. Therefore the algorithm will return with the list of the facets of $P^+$.  Second, we note that in $2$-dimensions the sewing is obvious: we just place the new vertex so that it is beyond the sewing edge and beneath all other edges.  Third we mention that each type (b) facet $F$ of $P^+$ contains either $x_1$ or $y_1$ or both. This fact follows from Figure \ref{torony}, since if $G$ is a $(2m-2)$-face of $P$ such that $F=[\bar x,G]$ does not contain $x_1$ nor $y_1$, then it obviously can't have the BBP. (More generally, see Lemma \ref{xisthere} below.) This means that all the type (b) facets of $P^+$ can be read from the list of facets of $P^+/[x_1,x]$ and $P^+/[y_1,x]$. Finally, the list of facets of $P/\Phi_i$ for $1\leq i \leq m-1$, can be derived easily by checking which facets of $P$ contain $\Phi_i$, hence we assume that $P/\Phi_1, \ldots, P/\Phi_{m-1}$ are also given.

	The algorithm is based on the fact that using Corollary \ref{main1} we can reduce the $2m$-dimensional sewing of $P$ to a $2m-2$-dimensional sewing, then we recover $P^+$ from the $2m-2$-dimensional sewn polytope. We apply this idea repeatedly, and so we start the sewing in $2$-dimensions, and then we do a sewing in $4$-dimensions based on the previous $2$-dimensional sewing, and so on until we obtain the list of the facets of $P^+$.

	Assume that $P$ is given by the list of the facets, and we are also given $\mathcal T$, a universal tower in $P$, containing $\Phi_1\subset\Phi_2\subset \ldots \subset \Phi_m$. We will use the following notation in the description of the algorithm. We will sew a new vertex $\bar z_i$ onto $P/\Phi_i$ through the tower $\mathcal T/\Phi_i$ (see Section \ref{mainsection}), and obtain the list of the facets of $(P/\Phi_i)^+$.
	
	\bigskip
	{\em \bf Algorithm 1}
	\smallskip
	
	\noindent{\em Step 1.} Sew $\bar z_{m-1}$ onto $P/\Phi_{m-1}$ to obtain the list of the facets of $(P/\Phi_{m-1})^+$
	\smallskip
	
	\noindent{\em Step 2.} For $k$ running from $2$ to $m$ do the following:
	\begin{enumerate}[label=(\roman{enumi})]
		\item For each facet $F$ of $P/\Phi_{m-k}$ if $\bar z_{m-k}$ is beneath $F$, then add $F$ to the list of the facets of $(P/\Phi_{m-k})^+$.
		\item For each facet $F^*$ of $(P/\Phi_{m-k-1})^+$ add $[F, x_{m-k-1}, \bar z_{m-k}]$ to the list of the facets of $(P/\Phi_{m-k})^+$. Since $P/\Phi_{m-k-1}$ is a quotient polytope of $P/\Phi_{m-k}$, the correspondence $F^*\to F$ is clear with the additional convention that $\bar z_{m-k-1}$ corresponds to $y_{m-k-1}$.
		\item For each facet $F^*$ of $P/\Phi_{m-k-1}$, if $\bar z_{m-k-1}$ is beneath $F^*$ then add $[F, y_{m-k-1}, \bar z_{m-k}]$ to the list of the facets of $(P/\Phi_{m-k})^+$.
	\end{enumerate}

	\smallskip
	The correctness of the algorithm follows from Corollary \ref{main1} and Theorem \ref{bbp} and from our preliminary remarks. We are left to determine the running time. The first thing we note here is that there are exactly $m$ sewings. It's easy to see that ``in each dimension'' the program spends linear time in the number of facets $f$ of $P$. From these facts it follows that the algorithm has $c(m)\cdot f$ running time, where $c(m)$ is a constant depending only on the dimension of $P$. If we consider the dimension to be fixed, then we obtain that this is a linear algorithm in $f$.

	\begin{remark}
	      It is well known that a $(2m)$-dimensional neighbourly polytope with $n$ vertices has \[f=\binom{n-m}{m}+\binom{n-m-1}{m-1}\] facets. Since $n\geq 2m+3$ is assumed (otherwise we obtain a cyclic polytope), hence $f>m^m\gg m$.
	\end{remark}

	\begin{remark}
	      If $n>3m$ then it is not hard to prove that $c(m)$ depends linearly on $m$.
	\end{remark}

\section{Keeping track of universal faces}

	In this section we give a complete picture of the odd dimensional universal faces using Theorem \ref{main}. With these results the algorithm given in Section \ref{algo} can be extended such that it keeps track of the universal faces during the sewing process. Note the list of the facets contains all information, however it is very time consuming to list all universal faces of a polytope given by the list of the facets.

	First we prove that the ``new'' universal faces of $P^+$ necessarily intersect the sewing edge $\Phi_1$.

	\begin{prop}\label{php}
		Let $C_{2\ell+2}$ be a cyclic $(2\ell)$-polytope with $2\ell+2$ vertices, and let $c_1,\ldots,c_{2k-1}\in\mathcal{V}(C_{2\ell+2})$ be distinct with $1\leq k\leq \ell+1$.  Then there is a subset $\{c_{i_1},\ldots,c_{i_k}\}\subset\{c_1,\ldots,c_{2k-1}\}$ and an $M\in\mathcal{M}(C_{2\ell+2})$ such that $\{c_{i_1},\ldots,c_{i_k}\}\subset M$.
	\end{prop}

	\begin{proof}
		Recall that $C_{2\ell+2}$ has exactly two missing faces (see \cite{shem2}, Remark 1.3): they are disjoint and each have $\ell+1$ vertices.  By the Pigeon Hole Principle, one of them contains at least $k$ of $c_1,\ldots,c_{2k-1}$.
	\end{proof}

	\begin{lemma}\label{xisthere}
		Let $U\in\mathcal{U}_{2k-1}(P^+)$ and assume $\bar x\in U$.  Then $U\cap\Phi_1\not = \emptyset$.
	\end{lemma}

	\begin{proof}
		Suppose on the contrary that $U\cap\Phi_1 = \emptyset$.  Let $U=[\bar x, v_1, \ldots, v_{2k-1}]$. Since $U\in\mathcal{U}_{2k-1}(P^+)$, Proposition \ref{missuniverse} implies that 
		\begin{equation} \label{elso}
			|U\cap M|\leq k \mbox{ for all } M \in \mathcal M(P^+).
		\end{equation} 
		From Lemma \ref{sewnmissingface} it follows that 
		\begin{equation}\label{masodik}
			\{\bar x\}\cup A\in \mathcal M(P^+) \mbox{ for every } A^*\in \mathcal M(P/\Phi_1). 
		\end{equation}
		From (\ref{elso}) and (\ref{masodik}) we obtain that
		\begin{equation}\label{harmadik}
			|A^*\cap\{v_1^*, \ldots, v_{2k-1}^*\}|\leq k-1 \mbox{ for every } A^*\in \mathcal M(P/\Phi_1).
		\end{equation}
		Consider any distinct $w_1^*, \ldots, w_{2m-2k+1}^*\in \mathcal V(P/\Phi_1)\setminus \{v_1^*, \ldots, v_{2k-1}^*\}$, and take $Q^*=[v_1^*, \ldots, v_{2k-1}^*,w_1^*, \ldots, w_{2m-2k+1}^*]$. Proposition \ref{php} yields that there exist an $\bar A^*\in \mathcal M(Q^*)$ and $\{v_{i_1}^*, \ldots, v_{i_k}^*\}\subset\{v_1^*, \ldots, v_{2k-1}^*\}$ such that $\{v_{i_1}^*, \ldots, v_{i_k}^*\}\subset\bar A^*$. Observe that since  $P/\Phi_1$ is neighbourly, $\bar A^*\in \mathcal M(P/\Phi_1)$. Obviously $|\bar A^* \cap \{v_1^*, \ldots, v_{2k-1}^*\}|\geq k$, which contradicts (\ref{harmadik}).
	\end{proof}

	Second we prove that those universal faces of $P^+$ that were already universal faces in $P$ can be characterized in terms of the ``new'' universal faces of $P^+$.

	\begin{lemma}\label{iiseven}
		Let $U\in\mathcal{U}_{2k-1}(P)$, with $\Phi_i\subseteq U$ and $x_{i+1}\notin U$.  If $i$ is odd, then $U\notin\mathcal{U}_{2k-1}(P^+)$.
	\end{lemma}

	\begin{proof}
		From Lemma \ref{sewnmissingface}, we have that for every $A^*\in\mathcal{M}(P/\Phi_{i+1})$, $M=\{x_1,y_1,x_3,y_3,\ldots,x_i,y_i,A\}\in\mathcal{M}(P^+)$. Note that $|\{x_1,y_1,x_3,y_3,\ldots,x_i,y_i\}|=i+1$.  Suppose that \begin{equation}\label{akarmi} |A\cap (\mathcal V(U)\setminus \mathcal V(\Phi_i))|\leq k-i-1 \mbox{ for all } A^*\in\mathcal{M}(P/\Phi_{i+1}),\end{equation} and seek a contradiction. Observe that $|\mathcal V(U)\setminus \mathcal V(\Phi_{i+1})|\geq  2k-2i-1$, since $\Phi_{i+1}\not\subseteq U$. Consider any distinct $v_1, \ldots, v_{2k-2i-1}\in \mathcal V(U)\setminus \mathcal V(\Phi_{i+1})$. We may argue similarly as in the proof of Lemma \ref{xisthere}: that there exists a $\bar B^*\in \mathcal M(P/\Phi_{i+1})$ that contains at least $k-i$ of $v_1^*, \ldots, v_{2k-2i-1}^*$, which contradicts (\ref{akarmi}).
	\end{proof}

	\begin{theorem}\label{universal}
		Let $U\in\mathcal{U}_{2k-1}(P)$, with $0\leq i< k < m$,  $\Phi_i\subseteq U$ and $x_{i+1}\notin U$. Then $U\in \mathcal U_{2k-1}(P^+)$ if, and only if, $i$ is even and $[U,\bar x, x_{i+1}]\in \mathcal U_{2k+1}(P^+)$.
	\end{theorem}

	\begin{proof}
		First assume that $U\in \mathcal U_{2k-1}(P^+)$. From Lemma \ref{iiseven} it follows that $i$ is even.  From Lemma \ref{towerleftovers} we obtain that $\Phi_i\in \mathcal U_{2i-1}(P^+)$ and $\Psi=[\Phi_i, \bar x, x_{i+1}]\in \mathcal U_{2i+1}(P^+)$.  Now, from Proposition \ref{univquot} we obtain that $[\bar x^*, x_{i+1}^*]=\Psi/\Phi_i\in \mathcal U_1(P^+/\Phi_i)$ and $U/\Phi_i \in \mathcal U_{2k-2i-1}(P^+/\Phi_i)$. From Theorem 3.4 in \cite{shem} it follows that $$\left [\Psi/\Phi_i, U/\Phi_i \right ] \in \mathcal U_{2k-2i+1}(P^+/\Phi_i).$$ In other words $[U, \bar x, x_{i+1}]/\Phi_i \in \mathcal U_{2k-2i+1}(P^+/\Phi_i)$, and using Proposition \ref{univquot} again we get that $[U,\bar x, x_{i+1}]\in \mathcal U_{2k+1}(P^+)$.

		Now, assume that $i$ is even and $[U,\bar x, x_{i+1}]\in \mathcal U_{2k+1}(P^+)$.  By Definition \ref{universalface} it is enough to prove that for any $S=\{v_1,\ldots, v_{m-k}\}\subset \mathcal V(P^+)$ we have that $[U,S]\in \mathcal B(P^+)$. We consider 4 cases. 

		\noindent{\bf Case I.} $S\cap\{\bar x, x_{i+1}\}=\emptyset$. Observe that $G=[U, S]\in \mathcal B(P)$ since $U\in \mathcal U_{2k-1}(P)$. We claim that there exists a facet $\bar F\in \mathcal F_{2m-1}(P)$ that contains $G$, and $\bar x$ is beneath $F$. Suppose on contrary that $\bar x$ is beyond every facet $F$ of $P$ that contains $G$. Since $\Phi_i\subset G$ and $i$ is even, it follows that $x_{i+1}\in F$ for all such $F$. This contradicts the well known fact that \[G=\bigcap_{\substack{F\in\mathcal{F}(P)\\F\supset G}} F.\]  Hence $\bar F$ is a facet of $P^+$, and since $P^+$ is simplicial, every subset of $\mathcal V(\bar F)$ determines a face of $P^+$. We obtain that $G\in \mathcal B(P^+)$. 

		\noindent{\bf  Case II.} $S\cap\{\bar x, x_{i+1}\}={\bar x}$. Write $v_{m-k}=\bar x$. Given that $[U, \bar x, x_{i+1}]\in \mathcal U_{2k+1}(P^+)$ we obtain that \[[U, \bar x, x_{i+1}, v_1, \ldots, v_{m-k-1}]\in \mathcal B(P^+).\]  Since $U\cup S\subset U\cup\{  \bar x, x_{i+1}, v_1, \ldots, v_{m-k-1}\}$ and $P^+$ is simplicial, it follows that $[U, S]\in \mathcal B(P^+)$.

		\noindent{\bf Case III.} $S\cap\{\bar x, x_{i+1}\}={x_{i+1}}$. Similar to Case II.

		\noindent{\bf Case IV.} $\{\bar x, x_{i+1}\}\subset S$. Similar to Case II.

		\noindent The proof of the theorem is complete.
	\end{proof}

	Lemma \ref{xisthere} and Theorem \ref{universal} together provide a fast way to keep track of the universal faces during the sewing process. Note that Lemma \ref{sewnmissingface} characterizes the missing faces of a sewn polytope, and Proposition \ref{missuniverse} describes the universal faces of the sewn polytope in terms of the missing faces. However, for practical reasons it is not efficient. The idea is the same as in Section \ref{algo}; first we introduce the neccessary notions, then we give an extended algorithm that also keeps track of the odd dimensional universal faces. In what follows we always assume that a polytope $P$ is given by the list of its facets, and in addition we are also given the list of all odd dimensional universal faces of $P$. As before, $U^*\in \mathcal U_{2k-2i-1}(P/\Phi_i)$ if and only if $[U, \Phi_i]\in \mathcal U_{2k-1}(P)$, hence we may assume that $P/\Phi_1, \ldots, P/\Phi_{m-1}$ are also given.

	\bigskip
	{\em \bf Algorithm 2 (extended algorithm)}
	\smallskip
	
	\noindent{\em Step 1.} Sew $\bar z_{m-1}$ onto $P/\Phi_{m-1}$ to obtain the list of the facets of $(P/\Phi_{m-1})^+$
	\smallskip
	
	\noindent{\em Step 2.} For $k$ running from $2$ to $m$ do the following:
	\begin{enumerate}[label=(\roman{enumi})]
		\item Obtain the list of the facets of $(P/\Phi_{m-k})^+=[P/\Phi_{m-k}, \bar z_{m-k}]$ as in Algorithm 1.
		\item For $j$ running from $k-1$ down to $1$ do the following
		\begin{enumerate}[label=(\alph{enumii})]
			\item For all $U\in \mathcal U_{2j-1}(P/\Phi_{m-k}),$ find the largest $i$ with $\Phi_{m-k+i}/\Phi_{m-k}\subset U$, and find a $v\in \mathcal V(\Phi_{m-k+i+1}/\Phi_{m-k})\setminus \mathcal V(\Phi_{m-k+i}/\Phi_{m-k})$. If $i$ is even and $[U, \bar z_{m-k},v]\in \mathcal U_{2j+1}((P/\Phi_{m-k})^+)$, then add $U$ to the list of universal $(2j-1)$-faces of $(P/\Phi_{m-k})^+$. (see Theorem \ref{universal})
			\item For each $U^*\in\mathcal U_{2j-3}(P/\Phi_{m-k-1})^+$, add $[U, x_{m-k-1}, \bar z_{m-k}]$ to the list of universal $(2j-1)$-faces of $(P/\Phi_{m-k})^+$. Since $P/\Phi_{m-k-1}$ is a quotient polytope of $P/\Phi_{m-k}$, the correspondence $U^*\to U$ is clear with the additional convention that $\bar z_{m-k-1}$ corresponds to $y_{m-k-1}$. ($\mathcal U_{-1}(P/\Phi_{m-k-1})^+=\{\emptyset\}$)
			\item For each $U^*\in\mathcal U_{2j-3}(P/\Phi_{m-k-1})^+$,  if $\bar z_{m-k-1}\notin \mathcal V(U^*)$, then add $[U, y_{m-k-1}, \bar z_{m-k}]$ to the list of universal $(2j-1)$-faces of $(P/\Phi_{m-k})^+$. ($\mathcal U_{-1}(P/\Phi_{m-k-1})^+=\{\emptyset\}$)
		\end{enumerate}
	\end{enumerate}

	\begin{table}[htpb]
		\centering
		\resizebox{\textwidth}{!}{
		\begin{tabular}{c||c|c|c|c|c|c|c|c|c|c|c}
			& $P/\Phi_{m-1}$ & & $P/\Phi_{m-2}$ & & $\cdots$ & & $P/\Phi_2$ & & $P/\Phi_1$ & & $P$\\
			& $\downarrow$ & & $\downarrow$ & & $\cdots$ & & $\downarrow$ & & $\downarrow$ & & $\downarrow$ \\
			Dimension & $(P/\Phi_{m-1})^+$ & & $(P/\Phi_{m-2})^+$ & & $\cdots$ & & $(P/\Phi_2)^+$ & & $(P/\Phi_1)^+$ & & $P^+$\\
			\hline
			$2m-1$ & & & & & & & & & & & $\frac{m^2-m+2}{2}$\\
			& & & & & & & & & & $\nearrow$ & $\downarrow$\\
			$2m-3$ & & & & & & & & & $\frac{m^2-3m+4}{2}$ & & $\frac{m^2-m+4}{2}$\\
			& & & & & & & & $\nearrow$ & $\downarrow$ & $\nearrow$ & $\downarrow$\\
			$\vdots$ & & & & & & & $\vdots$ & & $\vdots$ & & $\vdots$\\
			& & & & $\nearrow$ & $\cdots$ & $\nearrow$ & $\downarrow$ & $\nearrow$ & $\downarrow$ &  $\nearrow$ & $\downarrow$\\
			$3$ & & & $2$ & & $\cdots$ & & $\frac{m^2-3m}{2}$ & & $\frac{m^2-m+2}{2}$ & & $\frac{m^2+m-2}{2}$\\ 
			& & $\nearrow$ & $\downarrow$ & $\nearrow$ & $\cdots$ & $\nearrow$ & $\downarrow$ & $\nearrow$ & $\downarrow$ &  $\nearrow$ & $\downarrow$\\
			$1$ & $1$ & & $3$ & & $\cdots$ & & $\frac{m^2-3m+2}{2}$ & & $\frac{m^2-m}{2}$ & & $\frac{m^2+m}{2}$\\
		\end{tabular}}
		\caption{The order of the steps in Algorithm 2}\label{tabla}
	\end{table}

	\bigskip
	The numbers in Table~\ref{tabla} show the order of the steps in Algorithm 2, and the arrows show dependency. In the algorithm, the index variable $k$ shows in which column we are, and $j$ refers to the row. The correctness of Algorithm 2 follows from Lemma \ref{xisthere} and Theorem \ref{universal}. The running time is slightly worse than the running time of Algotirithm 1. We assume that the dimension $2m$ is fixed. As Table \ref{tabla} shows we do exactly $(m^2+m)/2$ steps, that is constant. Let \[n=\max_{j=1,\ldots, m} |\mathcal U_{2j-1}(P)|.\]  It's not hard to see, that in each step we spend $O(n\log n)$ time, so the cummulative running time is also $O(n \log n)$. The extra $\log n$ factor comes from checking the condition in Step. 2/(ii)/(a)
	
	\bigskip
	{\bf Acknowledgement.} We are grateful to professor Ted Bisztriczky for his useful remarks, hints and suggestions.

\bigskip
\small
\noindent Ryan Trelford\\
Department of Mathematics and Statistics, University of Calgary,\\
2500 University Drive NW 
Calgary, Alberta 
Canada T2N 1N4 \\
rgtrelfo@ucalgary.ca

\smallskip
\noindent Viktor V\'{\i}gh\\
Department of Mathematics and Statistics, University of Calgary,\\
2500 University Drive NW 
Calgary, Alberta 
Canada T2N 1N4 \\
vvigh@ucalgary.ca

\end{document}